\documentclass{amsart}
\usepackage{amssymb}
\usepackage[all,cmtip]{xy}
\input diagxy
\usepackage{hyperref}
\usepackage{todonotes}
\usepackage{qtree}

\newtheorem{theorem}{Theorem}[section]
\newtheorem{proposition}[theorem]{Proposition}
\newtheorem{corollary}[theorem]{Corollary}
\newtheorem{lemma}[theorem]{Lemma}

\theoremstyle{definition}

\newtheorem{definition}[theorem]{Definition}

\newtheorem{notation}[theorem]{Notation}

\numberwithin{equation}{section}

\theoremstyle{remark}
\newtheorem{remark}[theorem]{Remark}

\newcommand{\oMc}{\overline{\mathcal{M}}}
\newcommand{\noMc}{\mathcal{M}}
\newcommand{\oEc}{\overline{\mathcal{E}}}

\newcommand{\Cc}{\mathcal{C}}
\newcommand{\Dc}{\mathcal{D}}
\newcommand{\Ec}{\mathcal{E}}
\newcommand{\Hc}{\mathcal{H}}
\newcommand{\Oc}{\mathcal{O}}
\newcommand{\Pc}{\mathcal{P}}

\newcommand{\Nb}{\mathbb{N}}
\newcommand{\Pb}{\mathbb{P}}

\newcommand{\Zb}{\mathbb{Z}}

\newcommand{\Aut}{Aut}

\newcommand{\Spec}{Spec}
\newcommand{\res}{res}
\newcommand{\gl}{gl}

\DeclareMathOperator{\PGL}{PGL}

\begin{document}
\title{Modular Operads of Embedded Curves}
\author{Satoshi Kondo, Charles Siegel, and Jesse Wolfson}
\address{Faculty of Mathematics, Higher School of Economics, Moscow, Russia}
\email{skondo@hse.ru}
\address{Kavli IPMU (WPI), The University of Tokyo, Kashiwa, Chiba 277-8583, Japan}
\email{charles.siegel@ipmu.jp}
\address{Department of Mathematics, The University of Chicago, Chicago, USA}
\email{wolfson@math.uchicago.edu}

\begin{abstract}
    For each $k\ge 5$, we construct a modular operad $\oEc^k$ of ``$k$-log-canonically embedded'' curves. We also construct, for $k\ge 2$, a stable cyclic operad $\oEc^k_c$ of such curves, and, for $k\ge 1$, a cyclic operad $\oEc^k_{0,c}$ of ``$k$-log-canonically embedded'' rational curves.
\end{abstract}
\thanks{The first author was partially supported by Grant-in-Aid for Young Scientists (B) 30372577. The first and second authors
were partially supported by World Premier International Research Center Initiative (WPI Initiative), MEXT Japan. The third author was partially supported by an NSF Post-doctoral Research Fellowship under Grant No.\ DMS-1400349.}


\maketitle

\section{Introduction}
\begin{definition}
    Let $S$ be a scheme. We define a \emph{$k$-log-canonically embedded} stable marked curve
    $(\Cc,\{\sigma_i\}_{i=1}^n,\eta)$ over $S$ to be a stable marked curve $(\pi\colon\Cc\to S,\{\sigma_i\}_{i=1}^n)$, along with a projective embedding by a \emph{complete} linear system
    \begin{equation*}
        \eta\colon\Cc\to\Pb_S\left(\pi_\ast\omega_{\Cc/S}\left(\sum_{i=1}^n\sigma_i\right)^{\otimes k}\right)^\vee.
    \end{equation*}
    Isomorphisms of $k$-log-canonically embedded stable marked curves are defined in the natural manner.
\end{definition}
A pair of stable marked curves $(\Cc_1,\{\sigma_i\}_{i=1}^n)$ and $(\Cc_2,\{\tau_j\}_{j=1}^m)$ can be glued together to obtain a third such curve $(\Cc_1\cup_{\sigma_k\sim \tau_\ell}\Cc_2,\{\sigma_i,\tau_j\}_{i\neq k,j\neq\ell})$, for any choice of $k$ and $\ell$. Similarly, two points $\sigma_k$ and $\sigma_\ell$ on the same curve $(\Cc,\{\sigma_i\}_{i=1}^n)$ can be glued together to obtain a new curve $(\Cc/\sigma_k\sim\sigma_\ell,\{\sigma_i\}_{i\neq k,\ell})$. In this article, we construct analogous
gluings for $k$-log-canonically embedded curves. More conceptually, denote by $\oEc_{g,n}^k$ the moduli of $k$-log-canonically embedded stable curves of genus $g$ with $n$ marked points (see Definition \ref{def:loghilb}). For $k\ge 2$ (or $k\ge 1$ when $g_1=g_2=0$), we construct maps
\begin{equation}
    \oEc_{g_1,n_1+1}^k\times\oEc_{g_2,n_2+1}^k\to\oEc_{g_1+g_2,n_1+n_2}^k \label{eglue2curves}
\end{equation}
encoding the gluing of two embedded curves. For $k\ge 5$, we construct maps
\begin{equation}
    \oEc_{g,n+2}^k\to\oEc_{g+1,n}^k \label{eglue1curve}
\end{equation}
which encode gluing two points together on the same embedded curve. Our main result is now the following.
\begin{theorem}\label{thm:main}\mbox{}
    \begin{enumerate}
        \item For each $k\ge 5$, the maps \eqref{eglue2curves} and \eqref{eglue1curve} endow the collection $\{\oEc_{g,n}^k\}$ with the structure of a modular operad (in DM-stacks) which we denote $\oEc^k$.
        \item For $k\ge 1$, the maps \eqref{eglue2curves} endow the collection $\{\oEc_{0,n}^k\}$ with the structure of a cyclic operad (in schemes), which we denote $\oEc_{0,c}^k$.
        \item For each $k\ge 2$, the maps \eqref{eglue2curves} endow the collection $\{\oEc_{g,n}^k\}$ with the structure of a stable cyclic operad (in DM-stacks), which we denote $\oEc_c^k$.
    \end{enumerate}
    Further, the maps
    \begin{equation*}
        \oEc_{g,n}^k\to\oMc_{g,n}
    \end{equation*}
    given by forgetting the embedding determine maps
    \begin{enumerate}
        \item of modular operads (in DM-stacks)
            \begin{equation*}
                \oEc^k\to\oMc,
            \end{equation*}
        \item of cyclic operads (in schemes)
            \begin{equation*}
                \oEc_{0,c}^k\to\oMc_0,
            \end{equation*}
        \item and of stable cyclic operads (in DM-stacks)
            \begin{equation*}
                \oEc_c^k\to\oMc.
            \end{equation*}
    \end{enumerate}
\end{theorem}

We refer to the operads $\oEc_{0,c}^k$, $\oEc_c^k$ and $\oEc^k$ as the \emph{cyclic, stable cylic, and modular operads of $k$-log-canonically embedded curves}. These operads expand the small collection of examples of cyclic, stable cyclic and modular operads in schemes and DM-stacks. We were led to them by the analogy between $\oMc$ and the topological modular operad $\noMc_{top}$ of smooth, connected, oriented surfaces with boundary. Because the space of embeddings of a manifold $M$ in $\mathbb{R}^\infty$ is contractible, the operad $\noMc_{top}$ is equivalent to the operad (up to coherent homotopy) $\Ec^\infty_{top}$ of smooth, connected, oriented surfaces with boundary inside $\mathbb{R}^\infty$. This equivalence provides the starting point for many results on moduli of topological surfaces (e.g. Madsen and Weiss's proof of the Mumford conjecture \cite{MR2335797}). It is natural to ask whether one can similarly obtain information about the moduli of stable marked curves by studying moduli of embedded curves. We hope to pursue this in future work.

\subsection*{Acknowledgements}
We are very grateful to Kyoji Saito and Kavli IPMU for making this collaboration possible by hosting the third author during the 2013--2014 academic year. We thank the editors and the anonymous referees for helpful comments and suggestions.

\section{Preliminaries on Curves}
\begin{definition}[cf. \cite{MR702953}]
    Let $S$ be a scheme, and let $g$ and $n$ be non-negative integers such that $n\ge 3-2g$. A \emph{stable marked curve of genus $g$} over $S$, $(\Cc,\{\sigma_i\}_{i=1}^n)$, is a flat, projective morphism
    \begin{equation*}
        \pi\colon\Cc\to S,
    \end{equation*}
    of relative dimension 1, along with pairwise disjoint sections
    \begin{equation*}
        \sigma_i\colon S\to \Cc
    \end{equation*}
    for $i=1,\ldots,n$. We require that, for all geometric points $s$ of $S$,
    \begin{enumerate}
        \item the fibers $\Cc_s$ are reduced, connected curves with at most nodal singularities,
        \item the points $\sigma_i(s)$ lie in the smooth locus of $\Cc_s$ for all $i$,
        \item $h^1(\Cc_s,\Oc_{\Cc_s})=g$, and
        \item the normalization $\Cc_{a,s}^\nu$ of each irreducible component of $\Cc_s$ contains at least $3-2g_{a,s}$
            special
            points, where $g_{a,s}$ is the arithmetic genus of $\Cc_{a,s}^\nu$ and where a point is \emph{special} if it is
            either a point of the form $\sigma_i(s)$ or the pre-image of a node.
    \end{enumerate}
\end{definition}

Now let $\Cc\to S$ be a curve, i.e. a flat, projective morphism of relative dimension 1, not necessarily connected. We further
assume that for each geometric point $s$ of $S$, the fiber $\Cc_s$ has at most nodal singularities. Let $\sigma_1,\sigma_2\colon S\to \Cc$ be two disjoint sections such that for each geometric point $s$ of $S$, $\sigma_1(s)$ and $\sigma_2(s)$ lie in the
smooth locus of the fiber $\Cc_s$. Define $\Cc^{\gl}:=\Cc/\sigma_1\sim\sigma_2$, denote the quotient map by $\gl\colon\Cc\to\Cc^{\gl}$, and let $\sigma:=\gl\circ\sigma_1=\gl\circ\sigma_2$. Recall that for each line bundle $L$ on
$\Cc^{\gl}$ we have a short exact sequence
\begin{equation}
    0\to L \to \gl_*\gl^*L\to \sigma_\ast\sigma^\ast L\to 0.\label{normSES}
\end{equation}
Next, recall that the canonical line bundle $\omega_{\Cc/S}$ of a family of nodal curves, defined as $\det(\Omega^1_{\Cc/S})$,
admits the following description (cf. \cite[p.163]{MR702953}). Every section $\alpha$ of $\omega_{\Cc/S}$, when restricted to
the fiber $\Cc_s$ over a geometric point $s$ of $S$, is a rational 1-form $\alpha_s$ on the normalization of $\Cc_s$. Moreover, $\alpha_s$ has at most simple poles at the pre-images $\{p_{\pm,s}\}$ of the nodes $\{p_s\}$ and
\begin{equation*}
    \res_{p_{+,s}}\alpha_s+\res_{p_{-,s}}\alpha_s=0
\end{equation*}
for each node $p_s$ of the fiber $\Cc_s$. Along with Nakayama's Lemma, this implies that we have a canonical exact sequence of
$\Oc_{\Cc^{\gl}}$-modules
\begin{equation}
    0\to \omega_{\Cc^{\gl}/S}\to \gl_*\omega_{\Cc/S}(\sigma_1+\sigma_2)\to \sigma_\ast\Oc_S\to 0.\label{babycanSES}
\end{equation}
Choosing $L=\omega_{\Cc^{\gl}/S}$ and taking the obvious vertical maps from \eqref{normSES} to \eqref{babycanSES}, an
application of the 5-lemma tells us that $\gl_*\gl^*\omega_{\Cc^{\gl}/S}\cong\gl_*\omega_{\Cc/S}(\sigma_1+\sigma_2)$.

\begin{lemma}\label{lemma:canSES}
    In the situation above, let $D\subset\Cc^{\gl}$ be a divisor such that $D\to S$ is a flat map of degree $d$, and such that
    for each geometric point $s$ of $S$, the fiber $D_s$ is supported on the smooth locus of the fiber $\Cc^{\gl}_s$. Then for
    each $k\ge 1$, there is a short exact sequence
    \begin{equation*}
        0\to\omega_{\Cc^{\gl}/S}(D)^{\otimes k}\to\gl_*\omega_{\Cc/S}(D+\sigma_1+\sigma_2)^{\otimes k}\to\sigma_\ast\Oc_S\to 0.
    \end{equation*}
\end{lemma}
\begin{proof}
    If we take $L=\omega_{\Cc^{\gl}/S}(D)^{\otimes k}$, then \eqref{normSES} becomes
    \begin{equation*}
        0\to \omega_{\Cc^{\gl}/S}(D)^{\otimes k} \to \gl_*\gl^*\omega_{\Cc^{\gl}/S}(D)^{\otimes k}\to \sigma_\ast\Oc_S\to 0.
    \end{equation*}
    It remains to show that $\gl_*\gl^*\omega_{\Cc^{\gl}/S}(D)^{\otimes k}\cong\gl_*\omega_{\Cc/S}(D+\sigma_1+\sigma_2)^{\otimes  k}$.

    Using Nakayama's Lemma, it suffices to check that this isomorphism holds at each geometric point $s$ of $S$. Let
    $U_s\subset\Cc_s$ be an open set such that either both or neither of the points $\sigma_1(s)$ and $\sigma_2(s)$ are in
    $U_s$. So long as both $\gl^*\omega_{\Cc^{\gl}_s}(D_s)^{\otimes k}$ and $\omega_{\Cc_s}(D_s+\sigma_1(s)+\sigma_2(s))^{\otimes k}$ agree on every such $U_s$, the pushforwards will be isomorphic. By the above discussion, we see that
    \begin{equation*}
        \Gamma(U_s,\gl^*\omega_{\Cc^{\gl}_s}(D_s))\cong\Gamma(U_s,\omega_{\Cc_s}(D_s+\sigma_1(s)+\sigma_2(s)))
    \end{equation*}
    for each such $U_s$, and therefore
    \begin{equation*}
        \Gamma(U_s,\gl^*\omega_{\Cc^{\gl}_s}(D_s)^{\otimes
        k})\cong\Gamma(U_s,\omega_{\Cc_s}(D_s+\sigma_1(s)+\sigma_2(s))^{\otimes k})
    \end{equation*}
    as required.
\end{proof}

Now let $\Cc$ be a nodal curve over a field $\kappa$ and let $\nu\colon\Cc^\nu\to\Cc$ be its normalization. Recall that for any
line bundle $L$ on $\Cc$ we have a short exact sequence
\begin{equation}\label{normkSES}
    0\to L\to \nu_\ast\nu^\ast L\to\Oc_N\to 0
\end{equation}
analogous to \eqref{normSES}. By reasoning analogous to the proof of Lemma \ref{lemma:canSES}, we also have the following.
\begin{lemma}\label{lemma:normoverk}
    In the situation above, let $D$ be a divisor on $\Cc$. Denote by $N$ the divisor of nodes on $\Cc$, and denote by $P$ the
    divisor of pre-images of nodes in the normalization $\Cc^\nu$. Then for each $k\ge 1$, there is a short exact sequence of
    $\Oc_{\Cc}$-modules
    \begin{equation*}
        0\to\omega_{\Cc}(D)^{\otimes k}\to\nu_*\omega_{\Cc^\nu}(D+P)^{\otimes k}\to \Oc_N\to 0.
    \end{equation*}
\end{lemma}

\begin{proposition}[Riemann--Roch]\label{prop:RR}
    Let $\kappa$ be a field, and let $\Cc$ be a curve of arithmetic genus $g$ over $\kappa$, with at most nodal singularities.
    Let $L$ be a line bundle on $\Cc$ of total degree $d$. Then we have
    \begin{equation*}
        h^0(\Cc,L)-h^1(\Cc,L)=d-g+1.
    \end{equation*}
\end{proposition}
\begin{proof}
    Let $\nu:\Cc^\nu\to\Cc$ be the normalization of $\Cc$ and let $N$ be the divisor of nodes in $\Cc$. The sequence
    \eqref{normkSES} gives a long exact sequence on cohomology
    \begin{equation*}
        0\to H^0(\Cc,L)\to H^0(\Cc^\nu,\nu^*L)\to H^0(N,\Oc_N)\to H^1(\Cc,L)\to H^1(\Cc^\nu,\nu^*L)\to 0.
    \end{equation*}
    Exactness then implies that
    \begin{equation*}
        h^0(\Cc,L)-h^0(\Cc^\nu,\nu^*L)+j-h^1(\Cc,L)+h^1(\Cc^\nu,\nu^*L)=0,
    \end{equation*}
    where $j$ is the length of $N$.  We can rearrange terms and apply the smooth Riemann--Roch theorem:
    \begin{eqnarray*}
        h^0(\Cc,L)-h^1(\Cc,L)&=&h^0(\Cc^\nu,\nu^*L)-h^1(\Cc^\nu,\nu^*L)-j\\
        &=&\sum_a d_a-\sum_a g_a+\ell-j
    \end{eqnarray*}
    where $\ell$ is the number of irreducible components $C^\nu_a$ of $\Cc^\nu$, $d_a$ is the degree of $L$ restricted to the
    component $\Cc^\nu_a$, and $g_a$ is the geometric genus of $\Cc^\nu_a$. Using that $d=\sum_a d_a$ and $g=\sum_a g_a
    -(\ell-1)+j$, we conclude the result.
\end{proof}

\begin{lemma}\label{lemma:dimH0}
    Let $(\Cc,\{\sigma_i\}_{i=1}^n)$ be a stable marked curve over a field $\kappa$. Then, for $k\ge 2$, we have
    \begin{equation*}
        h^0\left(\Cc,\omega_{\Cc}\left(\sum_{i=1}^n\sigma_i\right)^{\otimes k}\right)=(2k-1)(g-1)+kn.
    \end{equation*}
    When $\Cc$ has arithmetic genus 0, the same formula holds for $k\ge 1$.
\end{lemma}
\begin{proof}
    Stability and $k\geq 2$ (or $k\ge 1$ for genus 0) imply that $\omega_{\Cc}\otimes\left(\omega_{\Cc}\left(\sum_{i=1}^n\sigma_i\right)^{\otimes
    -k}\right)$ has negative degree on each component of $\Cc$, and thus
    \begin{equation*}
        H^1\left(\Cc,\omega_{\Cc}\left(\sum_{i=1}^n\sigma_i\right)^{\otimes k}\right)=0.
    \end{equation*}
    If $\Cc$ is smooth (or even just irreducible), then by Riemann--Roch, we have
    \begin{align*}
        h^0\left(\Cc,\omega_{\Cc}\left(\sum_{i=1}^n\sigma_i\right)^{\otimes k}\right)&=k(2g-2+n)-g+1\\
        &=(2k-1)(g-1)+nk.
    \end{align*}
    For non-smooth $\Cc$, let $N$ be the divisor of nodes of $\Cc$ and let $j$ be the length of $N$. Let $\nu:\Cc^\nu\to\Cc$ be
    a normalization of $\Cc$, and let $P$ be the divisor of pre-images of the nodes. Because
    $H^1\left(\Cc,\omega_{\Cc}\left(\sum_{i=1}^n\sigma_i\right)^{\otimes k}\right)=0$ (as we showed above), Lemma
    \ref{lemma:normoverk} shows that we have a short exact sequence
    \begin{equation}\label{degses}
        \begin{xy}
            \morphism<750,0>[0`H^0\left(\Cc,\omega_{\Cc}\left(\sum_{i=1}^n\sigma_i\right)^{\otimes k}\right);]
            \morphism(750,0)<1350,0>[H^0\left(\Cc,\omega_{\Cc}\left(\sum_{i=1}^n\sigma_i\right)^{\otimes
            k}\right)`H^0\left(\Cc^\nu,\omega_{\Cc^\nu}\left(\sum_{i=1}^n\sigma_i+P\right)^{\otimes k}\right);]
            \morphism(2100,0)<900,0>[H^0\left(\Cc^\nu,\omega_{\Cc^\nu}\left(\sum_{i=1}^n\sigma_i+P\right)^{\otimes
            k}\right)`\kappa^j;]
            \morphism(3000,0)<300,0>[\kappa^j`0;]
        \end{xy}.
    \end{equation}
    Write $\Cc^\nu$ as a union of its irreducible components $\Cc^\nu=\bigcup_{a=1}^\ell \Cc^\nu_a$. Denote by
    $\{\sigma_{(a,i)}\}$ the set of marked points on the component $\Cc^\nu_a$, and define $n_a:=|\{\sigma_{(a,i)}\}|$. Denote
    by $g_a$ the geometric genus of $\Cc^\nu_a$.  Denote by $P_a$ the restriction of $P$ to $\Cc^\nu_a$, and define
    $p_a:=\deg(P_a)$. Then:
    \begin{eqnarray*}
        h^0\left(\Cc,\omega_{\Cc}\left(\sum_{i=1}^n\sigma_i\right)^{\otimes
        k}\right)&=&h^0\left(\Cc^\nu,\omega_{\Cc^\nu}\left(\sum_{i=1}^n\sigma_i+P\right)^{\otimes k}\right)-j,\\
        h^0\left(\Cc^\nu,\omega_{\Cc^\nu}\left(\sum_{i=1}^n\sigma_i+P\right)^{\otimes k}\right)&=&\sum_{a=1}^\ell
        h^0\left(\Cc^\nu_a,\omega_{\Cc^\nu_a}\left(\sum_{(a,i)}\sigma_{(a,i)}+P_a\right)^{\otimes k}\right),
    \end{eqnarray*}
    and
    \begin{align*}
        h^0\left(\Cc^\nu_a,\omega_{\Cc^\nu_a}\left(\sum_{(a,i)}\sigma_{(a,i)}+P_a\right)^{\otimes
        k}\right)&=\deg\left(\omega_{\Cc^\nu_a}\left(\sum_{(a,i)}\sigma_{(a,i)}+P_a\right)^{\otimes k}\right)-g_a+1\\
        &=k(2g_a-2+n_a+p_a)-g_a+1\\
        &=(2k-1)(g_a-1)+k(n_a+p_a).
    \end{align*}
    Substituting back, we get
    \begin{align*}
        h^0\left(\Cc^\nu,\omega_{\Cc^\nu}\left(\sum_{i=1}^n\sigma_i+P\right)^{\otimes k}\right)&=\sum_{a=1}^\ell
        h^0\left(\Cc^\nu_a,\omega_{\Cc^\nu_a}\left(\sum_{(a,i)}\sigma_{(a,i)}+P_a\right)^{\otimes k}\right)\\
        &=\sum_{a=1}^\ell ((2k-1)(g_a-1)+k(n_a+p_a))\\
        &=(2k-1)\sum_{a=1}^\ell (g_a-1)+k\sum_{a=1}^\ell (n_a+p_a)\\
        &=(2k-1)\left(\sum_{a=1}^\ell g_a-\ell\right)+k(n+2j).
    \end{align*}
    Using that $g=\sum_{a=1}^\ell g_a-(\ell-1)+j$, we have
    \begin{equation*}
        h^0\left(\Cc^\nu,\omega_{\Cc^\nu}\left(\sum_{i=1}^n\sigma_i+P\right)^{\otimes k}\right)=(2k-1)(g-1)+j+kn.
    \end{equation*}
    In light of the exact sequence \eqref{degses}, this implies the result.
\end{proof}

\section{Gluing Embedded Curves}\label{sec:glue}
\begin{definition}\label{def:markedembed}
    Let $S$ be a scheme. A \emph{marked, $k$-log canonically embedded curve over $S$} consists of the data
    $(\Cc,\{\sigma_i\}_{i=1}^n,\eta)$, where
    \begin{enumerate}
        \item $\pi\colon\Cc\to S$ is a flat, projective morphism of relative dimension 1,
        \item the pair $(\Cc,\{\sigma_i\})$ is a disjoint union of stable marked curves over $S$,
        \item $\eta$ denotes a projective embedding over $S$ by a \emph{complete} linear system
            \begin{equation*}
                \eta\colon\Cc\to\Pb_S\left(\pi_\ast\omega_{\Cc/S}\left(\sum_{i=1}^n\sigma_i\right)^{\otimes k}\right)^\vee.
            \end{equation*}
    \end{enumerate}
\end{definition}

Our goal in this section is to prove the following.
\begin{theorem}[Gluing Embedded Curves]\label{thm:glue}
    Let $S$ be a scheme. Let $(\Cc,\{\sigma_i\}_{i=1}^n,\eta)$ be a marked, $k$-log-canonically embedded
    curve over $S$. Denote by $\ell_{\sigma_1,\sigma_2}$ the line in $\Pb_S\left(\pi_\ast\omega_{\Cc/S}\left(\sum_{i=1}^n\sigma_i\right)^{\otimes k}\right)^\vee$ spanned by $\sigma_1$ and
    $\sigma_2$. If $k\ge 5$, then:
    \begin{enumerate}
        \item \label{gluesec} There exists a section
            \begin{equation*}
                \gamma\colon S\to \ell_{\sigma_1,\sigma_2}
            \end{equation*}
            depending functorially in $S$.
        \item  \label{glueemb} The projection from $\gamma$ gives an embedding
            \begin{equation*}
                \Cc^{\gl}:=\Cc/\sigma_1\sim \sigma_2\to^{\eta^{\gl}} \Pb_S\left(\pi_\ast\omega_{\Cc^{\gl}/S}\left(\sum_{i=3}^n\sigma_i\right)^{\otimes k}\right)^\vee.
            \end{equation*}
    \end{enumerate}
    If $\sigma_1$ and $\sigma_2$ live on different connected components of $\Cc$, then the claims hold for $k\ge 2$. If, in addition, all components of $\Cc$ have arithmetic genus 0, then the claims hold for $k\ge 1$.
\end{theorem}
\begin{remark}
    We choose $\sigma_1$ and $\sigma_2$ for notational convenience. Our proof applies equally well to any choice of $i$ and $j$.
\end{remark}

\begin{proof}
    To prove the theorem, we need to establish the claims \ref{gluesec} and \ref{glueemb}.

    \subsubsection*{1. Constructing the Section}
    \begin{lemma}[Claim \ref{gluesec}]\label{lemma:gluesec}
        Let $S$ be a scheme, let $k\ge 2$, and let $(\Cc,\{\sigma_i\}_{i=1}^n,\eta)$ be a marked, $k$-log-canonically embedded curve
        over $S$. Denote by $\ell_{\sigma_1,\sigma_2}$ the line in $\Pb_S\left(\pi_\ast\omega_{\Cc/S}\left(\sum_{i=1}^n\sigma_i\right)^{\otimes k}\right)^\vee$ spanned by $\sigma_1$ and $\sigma_2$. Then there exists a section
        \begin{equation*}
            \gamma\colon S\to \ell_{\sigma_1,\sigma_2},
        \end{equation*}
        depending functorially in $S$.

        If all irreducible components of $\Cc$ have arithmetic genus 0, and if $\sigma_1$ and $\sigma_2$ lie on different connected components, then we can take $k\ge 1$.
    \end{lemma}
    \begin{proof}
        Define $\Cc^{\gl}:=\Cc/\sigma_1\sim\sigma_2$.  Denote the quotient map by $\gl\colon\Cc\to\Cc^{\gl}$, and let
        $\sigma:=\gl\circ\sigma_1=\gl\circ\sigma_2$. By Lemma \ref{lemma:canSES}, we have a short exact sequence
        \begin{equation*}
            \begin{xy}
                \morphism<700,0>[0`\omega_{\Cc^{\gl}/S}\left(\sum_{i=3}^n\sigma_i\right)^{\otimes k};]
                \morphism(700,0)<1200,0>[\omega_{\Cc^{\gl}/S}\left(\sum_{i=3}^n\sigma_i\right)^{\otimes
                k}`\gl_*\omega_{\Cc/S}\left(\sum_{i=1}^n\sigma_i\right)^{\otimes k};]
                \morphism(1900,0)<850,0>[\gl_*\omega_{\Cc/S}\left(\sum_{i=1}^n\sigma_i\right)^{\otimes k}`\sigma_\ast\Oc_S;]
                \morphism(2750,0)<450,0>[\sigma_\ast\Oc_S`0;]
            \end{xy}.
        \end{equation*}
        Pushing this sequence forward to $S$ along the projection $\pi^{\gl}\colon\Cc^{\gl}\to S$, we obtain a long exact
        sequence
        \begin{align*}
            \begin{xy}
                \morphism(0,300)<700,0>[0`\pi^{\gl}_*\omega_{\Cc^{\gl}/S}\left(\sum_{i=3}^n\sigma_i\right)^{\otimes k};]
                \morphism(700,300)<1250,0>[\pi^{\gl}_*\omega_{\Cc^{\gl}/S}\left(\sum_{i=3}^n\sigma_i\right)^{\otimes
                k}`\pi^{\gl}_*\gl_\ast\omega_{\Cc/S}\left(\sum_{i=1}^n\sigma_i\right)^{\otimes k};]
                \morphism(1950,300)<1000,0>[\pi^{\gl}_*\gl_\ast\omega_{\Cc/S}\left(\sum_{i=1}^n\sigma_i\right)^{\otimes
                k}`\pi^{\gl}_*\sigma_\ast\Oc_S;]
                \morphism(500,0)<750,0>[`R^1\pi^{\gl}_*\omega_{\Cc^{\gl}/S}\left(\sum_{i=3}^n\sigma_i\right)^{\otimes k};]
                \morphism(1250,0)<1350,0>[R^1\pi^{\gl}_*\omega_{\Cc^{\gl}/S}\left(\sum_{i=3}^n\sigma_i\right)^{\otimes
                k}`R^1\pi^{\gl}_*\gl_*\omega_{\Cc/S}\left(\sum_{i=1}^n\sigma_i\right)^{\otimes k};]
                \morphism(2600,0)<750,0>[R^1\pi^{\gl}_*\gl_*\omega_{\Cc/S}\left(\sum_{i=1}^n\sigma_i\right)^{\otimes k}`0;]
            \end{xy}.
        \end{align*}
        Because $k\ge 2$ (or $k\ge 1$ if $\Cc$ and $\Cc^{\gl}$ have arithmetic genus 0), degree considerations combine with Grothendieck--Riemann--Roch to show that the higher direct image
        sheaves vanish. Because $\pi^{\gl}\sigma=1_S$ and $\pi^{\gl}\gl=\pi$, we can rewrite the cohomology long exact sequence
        as the short exact sequence
        \begin{equation*}
            \begin{xy}
                \morphism<700,0>[0`\pi^{\gl}_*\omega_{\Cc^{\gl}/S}\left(\sum_{i=3}^n\sigma_i\right)^{\otimes k};]
                \morphism(700,0)<1200,0>[\pi^{\gl}_*\omega_{\Cc^{\gl}/S}\left(\sum_{i=3}^n\sigma_i\right)^{\otimes k}`
                \pi_\ast\omega_{\Cc/S}\left(\sum_{i=1}^n\sigma_i\right)^{\otimes k};]
                \morphism(1900,0)<750,0>[\pi_\ast\omega_{\Cc/S}\left(\sum_{i=1}^n\sigma_i\right)^{\otimes k}`\Oc_S;]
                \morphism(2650,0)<400,0>[\Oc_S`0;]
            \end{xy}.
        \end{equation*}
        Dualizing and projectivizing, we obtain the sequence
        \begin{equation*}
            \begin{xy}
                \morphism<850,0>[S`\Pb_S\left(\pi_\ast\omega_{\Cc/S}\left(\sum_{i=1}^n\sigma_i\right)^{\otimes k}\right)^\vee;\gamma]
                \morphism(850,0)/-->/<1500,0>[\Pb_S\left(\pi_\ast\omega_{\Cc/S}\left(\sum_{i=1}^n\sigma_i\right)^{\otimes
                k}\right)^\vee`\Pb_S\left(\pi^{\gl}_*\omega_{\Cc^{\gl}/S}\left(\sum_{i=3}^n\sigma_i\right)^{\otimes k}\right)^\vee;]
            \end{xy}
        \end{equation*}
        where the dashed arrow indicates the projection from the point $\gamma$.

        The first map gives the desired section
		\begin{equation*}
            \begin{xy}
                \morphism<850,0>[S`\Pb_S\left(\pi_\ast\omega_{\Cc/S}\left(\sum_{i=1}^n\sigma_i\right)^{\otimes k}\right)^\vee;\gamma]
            \end{xy}.
		\end{equation*}
        We must still show that $\gamma$ factors through $\ell_{\sigma_1,\sigma_2}$.  We have the map
		\begin{equation*}
            \begin{xy}
                \morphism<850,0>[\Cc`\Pb_S\left(\pi_\ast\omega_{\Cc/S}\left(\sum_{i=1}^n\sigma_i\right)^{\otimes k}\right)^\vee;\eta]
                \morphism(850,0)/-->/<1500,0>[\Pb_S\left(\pi_\ast\omega_{\Cc/S}\left(\sum_{i=1}^n\sigma_i\right)^{\otimes
                k}\right)^\vee`\Pb_S\left(\pi^{\gl}_*\omega_{\Cc^{\gl}/S}\left(\sum_{i=3}^n\sigma_i\right)^{\otimes k}\right)^\vee;]
            \end{xy}
		\end{equation*}
        which comes from restricting the linear system $\pi_\ast\omega_{\Cc/S}\left(\sum_{i=1}^n\sigma_i\right)^{\otimes k}$ to
        the sections in $\pi^{\gl}_*\omega_{\Cc^{\gl}/S}\left(\sum_{i=3}^n\sigma_i\right)^{\otimes k}$. These sections, by
        construction, agree on $\sigma_1$ and $\sigma_2$.  Thus, this composition factors through $\Cc^{\gl}$, and $\gamma$
        factors through $\ell_{\sigma_1,\sigma_2}$.
    \end{proof}

    \subsubsection*{2. Verifying that Projecting Gives an Embedding}
    It remains to show that projecting from $\gamma$ induces an embedding $\eta^{\gl}$ of the glued curve $\Cc^{\gl}$. Because
    the projection from $\gamma$ is a map over $S$, it suffices to check that it gives an embedding on fibers. Therefore,
    throughout this section, we assume that $S=\Spec(\kappa)$ for a field $\kappa$.

    The following proposition provides the basis for our approach.
    \begin{proposition}\cite[Proposition IV.3.7]{Har:77}\label{IV.3.7}
        Let $\kappa$ be a field, let $\Cc$ be a curve in $\Pb^3_\kappa$, let $\gamma$ be a $\kappa$-point not on $\Cc$, and let
        $\eta':\Cc\to \Pb^2_\kappa$ be the morphism determined by projection from $\gamma$.  Then $\eta'$ is birational onto its
        image and $\eta'(\Cc)$ has at most nodes as singularities if and only if:
        \begin{enumerate}
             \item $\gamma$ lies on only finitely many secants of $\Cc$,
             \item $\gamma$ is not on any tangent line of $\Cc$,
             \item $\gamma$ is not on any secant with coplanar tangent lines, and
             \item $\gamma$ is not on any multisecant of $\Cc$.
        \end{enumerate}
    \end{proposition}

    Now let $(\Cc,\{\sigma_i\}_{i=1}^n,\eta)$ be a marked $k$-log-canonically embedded curve over the field $\kappa$. We
    show that any $\kappa$-point $\gamma\in\ell_{\sigma_1,\sigma_2}$ on the line spanned by $\sigma_1$ and $\sigma_2$ satisfies
    an analogue of Proposition \ref{IV.3.7}. As a first step, we have:
    \begin{lemma}\label{lemma:disjointglue}
        Let $k\ge 2$, and let $(\Cc,\{\sigma_i\}_{i=1}^n,\eta)$ be a marked $k$-log-canonically embedded curve over the field $\kappa$, and suppose $\Cc$ is a disjoint union $\Cc_1\sqcup\Cc_2$ with $\sigma_1\in\Cc_1$ and $\sigma_2\in\Cc_2$. Then the projection from a point
        $\gamma\in\ell_{\sigma_1,\sigma_2}\setminus\{\sigma_1,\sigma_2\}$ is an isomorphism on $\Cc\setminus
        \{\sigma_1,\sigma_2\}$ with a nodal singularity at $\sigma=\ell_{\sigma_1,\sigma_2}$. Further, if all components of $\Cc$ have arithmetic genus 0, then the result holds for $k\ge 1$.
    \end{lemma}
    \begin{proof}
        By Lemma \ref{lemma:gluesec}, our assumptions on $\Cc$ and $k$ guarantee the existence of $\gamma$. The fibers of the projection restricted to $\Cc$ are intersections with lines through $\gamma$. In other words, any line through $\gamma$ intersecting $\Cc$ in more than one point is a fiber where the map is non-injective. Thus, to show the map is injective, it suffices to show that $\gamma$ lies on a unique secant line of $\Cc$.  To see that it is an isomorphism, we
        note that, if $\gamma$ lies on a unique secant, then the projection from $\gamma$ will be an isomorphism at any point where the line intersects the curve transversely, so we just need to rule out the existence of tangent lines to $\Cc$ containing $\gamma$.

        We first observe that any $k$-log-canonical embedding of $\Cc=\Cc_1\sqcup\Cc_2$ will embed $\Cc_i$ in disjoint projective subspaces, which, up to a projective linear transformation, we can take to be
        \begin{equation*}
            \Pb(H^0(\Cc_a,\omega_{\Cc_a}(D_a)^{\otimes k}))\subset \Pb(H^0(\Cc,\omega_{\Cc}(\sum_i\sigma_i)^{\otimes k})),
        \end{equation*}
        where $D_a$ denotes the divisor $\Cc_a\cap\sum_i\sigma_i$ for $a=1,2$. From this, we immediately see that, if $\gamma\in\ell_{\sigma_1,\sigma_2}\setminus\{\sigma_1,\sigma_2\}$, then $\gamma$ is not contained in $\Pb(H^0(\Cc_a,\omega_{\Cc_a}(D_a)^{\otimes k}))$ for $a=1,2$.  Therefore, $\gamma$ does not lie on any tangent line of $\Cc$. Similarly, if $p,q\in\Cc$ are two points in $\Cc_a$, then the secant line $\ell_{p,q}$ is contained in $\Pb(H^0(\Cc_a,\omega_{\Cc_a}(D_a)^{\otimes k}))$ and therefore does not contain $\gamma$.  From this we also see that $\Cc$ has no multisecants connecting $\Cc_1$ and $\Cc_2$.  Therefore, for any four points $\sigma_1\neq p\in\Cc_1$ and $\sigma_2\neq q\in\Cc_2$ are in general position, and $\ell_{\sigma_1,\sigma_2}\cap\ell_{p,q}=\emptyset$. We conclude that $\ell_{\sigma_1,\sigma_2}$ is the unique secant containing $\gamma$.
    \end{proof}

    For $k\ge 5$ and no assumptions on $\sigma_1,\sigma_2$, we can prove a stronger statement than the conditions of Proposition \ref{IV.3.7}. A priori, it suffices to change the first requirement so that $\gamma$ lies on a unique secant $\ell_{\sigma_1,\sigma_2}$.  Using $\ell_{p,q}$ for the line between $p$ and $q$,we now rephrase (and strengthen) the four criteria as:
    \begin{enumerate}
        \item for all $p,q\in \Cc$, $\ell_{p,q}\cap\ell_{\sigma_1,\sigma_2}=\emptyset$ unless $\{p,q\}\cap
            \{\sigma_1,\sigma_2\}\neq \emptyset$,
        \item for all $p\in\Cc$, $T_p \Cc\cap \ell_{\sigma_1,\sigma_2}=\emptyset$ unless $p\in \{\sigma_1,\sigma_2\}$,
        \item for all $p,q\in \Cc$, $T_p\Cc\cap T_q\Cc=\emptyset$ unless $p=q$, and
        \item $\Cc$ has no multisecant.
    \end{enumerate}
    We can further simplify as follows. Recall that the \emph{length} of a zero-dimensional scheme $X$ over a field $\kappa$ is the dimension of the $\kappa$-vector space $H^0(X,\Oc_X)$. We now note that conditions 1--3 are special cases of the same thing. In particular, if we begin with condition 1, and take the limit as $q$ approaches $p$, we arrive at condition 2, and as $\sigma_1$
    approaches $\sigma_2$, we get 3.  Second, $\Cc$ has no multisecants if and only if $\Cc$ has no trisecants. With these changes, the statement becomes the following.
    \begin{proposition}\label{prop:embed}
        Let $\Cc$ be a curve in $\Pb^N_\kappa$, and let $\sigma_1,\sigma_2\in\Cc$. Then the projection from a point
        $\gamma\in\ell_{\sigma_1,\sigma_2}\setminus\{\sigma_1,\sigma_2\}$ is an isomorphism on $\Cc\setminus
        \{\sigma_1,\sigma_2\}$ if:
        \begin{enumerate}
            \item $\Cc\subset\Pb^N_\kappa$ has no trisecant.
            \item No length 4 sub-scheme of $\Cc$ is contained in a plane.
        \end{enumerate}
    \end{proposition}
	\begin{proof}
        As noted above, the fibers of the projection restricted to $\Cc$ are intersections with lines through $\gamma$. The absence of a        trisecant line guarantees that fibers consist of at most two points.  The lack of a quadrisecant plane guarantees that
        there is only one line through $\gamma$ which intersects the curve in at least two points.  Thus, away from
        $\ell_{\sigma_1,\sigma_2}$, the projection map is injective on the curve $\Cc$.  As noted above, to see that it is an isomorphism, it remains to rule out the existence of tangent lines to $\Cc$ containing $\gamma$.  Such a line would, along with
        $\ell_{\sigma_1,\sigma_2}$ give a length 4 sub-scheme of $\Cc$ contained in a plane (through $\gamma$); by hypothesis,
        none exists.
	\end{proof}

    We now verify that every marked $k$-log-canonically embedded curve satisfies the conditions of the proposition.
    \begin{lemma}\label{lemma:trisec}
        Let $\Cc=\bigcup\Cc_a$ be a nodal curve (with irreducible components $\Cc_a$) of arithmetic genus $g$ over a field
        $\kappa$.  Let $g_a$ be the geometric genus of the normalization $\Cc^\nu_a$ of $\Cc_a$.  Let $L$ be a line bundle of
        degree $d$ on $\Cc$, let $L_a$ be the pullback to $\Cc_a^\nu$ of $L$, and let $d_a:=\deg L_a$. Assume that, for all $a$,
        $d_a\geq 2g_a+2+j_a$, where $j_a$ is the number of preimages of nodes in $\Cc^\nu_a$. Then $\Cc$ has no trisecant lines
        when embedded by the complete linear system $|L|$.
    \end{lemma}
    \begin{proof}
        A trisecant is an effective divisor $T$ of degree 3 that is contained in a line. For a curve embedded by the complete
        linear system of a line bundle $L$, this condition on $T$ can be rewritten as $h^0(\Cc,L)-h^0(\Cc,L(-T))=2$.
        Riemann--Roch (Proposition \ref{prop:RR}) tells us that
        \begin{eqnarray*}
            h^0(\Cc,L)-h^1(\Cc,L)&=&d-g+1,\text{ and}\\
            h^0(\Cc,L(-T))-h^1(\Cc,L(-T))&=&d-g-2.
        \end{eqnarray*}
        Applying Serre duality and then subtracting one from the other, we get
        \begin{equation*}
            \left(h^0(\Cc,L)-h^0(\Cc,L(-T))\right)-\left(h^0(\Cc,\omega_{\Cc}\otimes L^{-1})-h^0(\Cc,\omega_{\Cc}(T)\otimes
            L^{-1})\right)=3.
        \end{equation*}
        This equation implies that, in order to show that $T$ is not a trisecant, it suffices to show that
        \begin{equation*}
            h^0(\Cc,\omega_{\Cc}\otimes L^{-1})-h^0(\Cc,\omega_{\Cc}(T)\otimes L^{-1})=0.
        \end{equation*}
        In particular, it suffices to show that both terms vanish. A line bundle can be shown to have no global sections by
        checking that there is no component on which the degree is positive. Thus, we want $2g_a-2+j_a-d_a<0$ and
        $2g_a-2+3+j_a-d_a<0$. We see that $d_a\ge 2g_a+2+j_a$ suffices for both.
    \end{proof}

    \begin{lemma}\label{lemma:quadrisec}
        In the situation of Lemma \ref{lemma:trisec}, assume that, for all $a$, $d_a\geq 2g_a+3+j_a$. Then, when embedded by the
        complete linear system $|L|$, $\Cc$ has no quadrisecant planes.
    \end{lemma}
    \begin{proof}
        Let $T$ be an effective divisor of degree 4 on $\Cc$ that is contained in a plane.  Then, by a similar calculation as in
        the proof of Lemma \ref{lemma:trisec}, the divisor $T$ must satisfy
        \begin{equation*}
            h^0(\Cc,\omega_{\Cc}(T)\otimes L^{-1})-h^0(\Cc,\omega_{\Cc}\otimes L^{-1})=1.
        \end{equation*}
        By the degree condition, $\omega_{\Cc}\otimes L^{-1}$ already has negative degree on each component, and so has no
        global
        sections. Therefore, $T$ is contained in a plane if and only if
        \begin{equation*}
            h^0(C,\omega_{\Cc}(T)\otimes L^{-1})=1.
        \end{equation*}
        However, we can compute that the degree on each component is
        \begin{equation*}
            2g_a-2+j_a+4-d_a=2g_a+2+j_a-d_a.
        \end{equation*}
        By hypothesis, this is negative.
    \end{proof}

    A direct computation now shows that if $(\Cc,\{\sigma_i\}_{i=1}^n)$ is a disjoint union of stable curves, then
    $\deg\left(\omega_{\Cc}\left(\sum_{i=1}^n\sigma_i\right)^{\otimes k}\right)$ satisfies the conditions of Lemma
    \ref{lemma:quadrisec} (and thus Lemma \ref{lemma:trisec}) so long as $k\ge 5$.
    \begin{corollary}[Claim \ref{glueemb}]
        Let $k\ge 5$. Let $S$ be a scheme, and let $(\Cc,\{\sigma_i\}_{i=1}^n,\eta)$ be a marked, $k$-log-canonically
        embedded curve over $S$ (as in Theorem \ref{thm:glue}). Then the projection from $\gamma$ in $\Pb_S\left(\pi_\ast\omega_{\Cc/S}\left(\sum_{i=1}^n\sigma_i\right)^{\otimes k}\right)^\vee$ induces an embedding $\eta^{\gl}$ of $\Cc^{\gl}$ in $\Pb_S\left(\pi_\ast\omega_{\Cc^{\gl}/S}\left(\sum_{i=3}^n\sigma_i\right)^{\otimes k}\right)^\vee$.
    \end{corollary}
    This concludes the proof of Theorem \ref{thm:glue}.
\end{proof}

\section{Moduli of Pluri-Log-Canonically Embedded Curves}
We now introduce a smooth DM stack $\oEc^k_{g,n}$ parameterizing stable curves $\Cc$, of genus $g$, with $n$ marked points $\{\sigma_i\}_{i=1}^n$ in the smooth locus of $\Cc$, equipped with a projective embedding $\eta$ by the complete linear system
$\omega_{\Cc}\left(\sum_{i=1}^n\sigma_i\right)^{\otimes k}$.

\begin{definition}\label{def:loghilb}
    Let $k\ge 2$. We define the \emph{moduli stack of $k$-log-canonically embedded marked curves} $\oEc_{g,n}^k$ to be the DM stack representing the functor
    \begin{equation*}
        S\mapsto\{(\Cc,\{\sigma_i\}_{i=1}^n,\eta)\}
    \end{equation*}
    which maps a scheme $S$ to the groupoid of stable curves of genus $g$ with $n$-marked points and a $k$-log-canonical embedding $\eta\colon \Cc\to \Pb_S(\pi_\ast\omega_{\Cc/S}\left(\sum_{i=1}^n \sigma_i\right)^{\otimes k})^\vee$. When $g=0$, we can take $k\ge 1$, in which case $\oEc_{0,n}^k$ is in fact a scheme.
\end{definition}

\begin{remark}
    We see that $\oEc_{g,n}^k$ is representable and smooth over $\Spec(\Zb)$ as follows. Denote by $V(g,n,k)\to\oMc_{g,n}$ the ``$k$-log-Hodge'' bundle, whose fiber at $(\Cc,\{\sigma_i\})$ is given by $H^0(\Cc,\omega_{\Cc}(\sum_i\sigma_i)^{\otimes k})^\vee$. Because any two projective embeddings by a complete linear system differ by a change of basis, $\oEc_{g,n}^k\to\oMc_{g,n}$ is a torsor for the relative group scheme $\PGL(V(g,n,k))\to\oMc_{g,n}$ of projective linear automorphisms of $V(g,n,k)$. Further, the torsor $\oEc_{g,n}^k$ has a section, given by sending a curve $\Cc$ to the embedding which sends a point $x\in\Cc$ to the hyperplane $H_x\subset H^0(\Cc,\omega_{\Cc}(\sum_i\sigma_i)^{\otimes k})$ consisting of sections vanishing at $x$.  Therefore the torsor trivializes, and $\oEc_{g,n}^k$ is isomorphic to $\PGL(V(g,n,k))$ over $\oMc_{g,n}$.  In less elementary fashion, one could also directly exhibit a smooth atlas for $\oEc_{g,n}^k$ by a construction analogous to the construction of $\oMc_{g,n}$ from the Hilbert scheme.
\end{remark}

Every $S$-point of $\oEc_{g,n+2}^k$ or $\oEc_{g_1,n_1+1}^k\times\oEc_{g_2,n_2+1}^k$ determines an embedded curve satisfying the
conditions of Theorem \ref{thm:glue}.\footnote{For $\oEc_{g_1,n_1+1}^k\times\oEc_{g_2,n_2+1}^k$, take the disjoint union of the
factors.} Because the section $\gamma$ in Theorem \ref{thm:glue} is natural with respect to base change, Theorem \ref{thm:glue}
immediately implies the following.
\begin{corollary}\label{cor:forget}
    For each $k\ge 2$ (and $k\ge 1$ for $g_1=g_2=0$), there exists a map
    \begin{align*}
        \oEc_{g_1,n_1+1}^k\times\oEc_{g_2,n_2+1}^k&\to\oEc_{g_1+g_2,n_1+n_2}^k.\intertext{For $k\ge 5$, there exists a map}
        \oEc_{g,n+2}^k&\to\oEc_{g+1,n}^k.
    \end{align*}
    These maps fit into commuting squares
    \begin{equation*}
        \begin{xy}
            \square(-1500,0)<1000,500>[\oEc_{g_1,n_1+1}^k\times\oEc_{g_2,n_2+1}^k`\oEc_{g_1+g_2,n_1+n_2}^k`\oMc_{g_1,n_1+1}\times\oMc_{g_2,n_2+1}`\oMc_{g_1+g_2,n_1+n_2};```]
            \place(0,250)[\text{and}]
            \square(500,0)<750,500>[\oEc_{g,n+2}^k`\oEc_{g+1,n}^k`\oMc_{g,n+2}`\oMc_{g+1,n};```]
        \end{xy}.
    \end{equation*}
\end{corollary}

\begin{remark}
    In the case $g=0,~k=1$, the discussion of Section \ref{sec:glue} could be carried out for embedded stable genus 0 curves $(\Cc,\{\sigma_i\}_{i=1}^n,\eta)$ which are further equipped with the canonical isomorphism
    \begin{equation*}
       \varphi\colon\Pb^{n-2}_S\to^\cong\Pb_S\left(\pi_\ast \omega_{\Cc/S}(\sum_{i=1}^n\sigma_i)\right)^\vee.
    \end{equation*}
    which sends the standard coordinate points of $\Pb^{n-2}$ to the $n$ points in general position $\{\sigma_i\}_{i=1}^n$ . We could then consider the functor
    \begin{equation*}
        S\mapsto\{(\Cc,\{\sigma_i\}_{i=1}^n,\eta,\varphi)\}
    \end{equation*}
    which maps an $S$-scheme to a stable, marked log-canonically embedded curve of genus 0 with the specified trivialization of the ambient projective bundle.  In the notation of the previous remark, this functor is represented by the closed subscheme $\Hc_{0,n}^1$ of the Hilbert scheme of $\Pb^{n-2}$. Kapranov \cite{MR1203685} has shown that the forgetful map $\Hc_{0,n}^1\to \oMc_{0,n}$ is an isomorphism.  Kapranov's work can thus be understood as a first instance of gluing maps for log-canonically embedded stable curves.

    For larger $g$ and $k$, one could similarly consider \emph{framed} $k$-log-canonically embedded stable marked curves, i.e. curves as above equipped with a collection of $(2k-1)(g-1)+nk+1$ points in general position containing the $n$ marked points as a subset. The constructions of Section \ref{sec:glue} extend to framed $k$-log-canonically embedded curves. However, we do not know, even in the genus 0 case for $k>1$, how to ensure that the gluing maps for framed curves satisfy the necessary equivariance properties required of a cyclic (and thus a modular) operad as required for the next section.
\end{remark}

\section{Modular Operads of Embedded Curves}
In this section, we prove Theorem \ref{thm:main}. For the reader familiar with modular operads, we remark that, given the above
construction of the gluing maps, the only non-trivial point which remains is to prove that the gluing maps are associative. For
the rest of our readers, we begin by recalling the definition of modular operads and stating what it is we need to show. Readers
familiar with these notions should feel free to skip the following paragraph.

\subsubsection*{Review of Modular Operads}
Our goal in this paragraph is to provide a minimal list of things one must produce to exhibit a modular operad. For a more
elegant and thorough treatment, we refer the reader to the article \cite{GeK:98}, which we take as our primary reference.

\begin{definition}
    Denote by $S_n$ the permutation group on $n$ elements $\{1,\ldots,n\}$.\footnote{By convention, $S_0$ is the trivial group,
    i.e. the group of automorphisms of the empty set.} For $\pi\in S_m$, $\rho\in S_n$, and $1\le i\le m$, denote by
    \begin{equation*}
        \pi\circ_i\rho\in S_{m+n-1}
    \end{equation*}
    the permutation which re-orders $\{i,\ldots,i+n-1\}$ according to $\rho$ and then re-orders the set of sets
    $\{\{1\},\ldots,\{i-1\},\{i,\ldots,i+n-1\},\{i+n\},\ldots,\{m+n-1\}\}$ by $\pi$. Explicitly,
    \begin{equation*}
        (\pi\circ_i\rho)(j):=\left\{
            \begin{array}{lr}
                \pi(j) & j<i\text{ and }\pi(j)<\pi(i),\\
                \pi(j)+n-1 & j<i\text{ and }\pi(j)>\pi(i),\\
                \pi(i)+\rho(j-i+1)-1 & i-1<j<i+n,\\
                \pi(j-n+1) & j\ge i+n\text{ and }\pi(j-n+1)<\pi(i),\\
                \pi(j-n+1)+n-1 & j\ge i+n\text{ and }\pi(j-n+1)>\pi(i).
            \end{array}\right.
    \end{equation*}
\end{definition}

Let $(\Dc,\otimes,\sigma)$ be a symmetric monoidal category\footnote{$\sigma$ denotes the symmetry isomorphism.} such that
$\otimes$ preserves coproducts, for example $\Dc$ could be the category of DM-stacks with the Cartesian product. We further
assume that there exists an initial object $0\in\Dc$; e.g. $0$ could be the DM-stack $\emptyset$.
\begin{definition}\label{def:operad}
    An operad in $(\Dc,\otimes,\sigma)$ consists of:
    \begin{enumerate}
        \item for each non-negative integer $n\in\Nb$, an object $\Pc(n)\in\Dc$ with a homomorphism $S_n\to\Aut(\Pc(n))$, and
        \item for each $1\le i\le m$, a map $\circ_i\colon\Pc(m)\otimes\Pc(n)\to\Pc(m+n-1)$.
    \end{enumerate}
    We require that these satisfy the following conditions. For $\pi\in S_m$, $\rho\in S_n$, and $1\le i\le m$, we require
    \begin{equation}
        (\pi\circ_i\rho)\cdot\circ_i=\circ_{\pi(i)}\cdot(\pi\otimes\rho)\label{opeq}
    \end{equation}
    as maps $\Pc(m)\otimes\Pc(n)\to\Pc(m+n-1)$.

    For $1\le i<j\le \ell$, we require
    \begin{align}
        \circ_{j+m-1}\cdot(\circ_i\otimes 1_{\Pc(n)})&=\circ_i\cdot(\circ_j\otimes
        1_{\Pc(m)})\cdot(1_{\Pc(\ell)}\otimes\sigma)\label{opass1}\intertext{and for $1\le i\le\ell$ and $1\le j\le m$, we
        require}
        \circ_{i+j-1}\cdot(\circ_i\otimes 1_{\Pc(n)})&=\circ_i\cdot(1_{\Pc(\ell)}\otimes\circ_j)\label{opass2}
    \end{align}
    as maps $\Pc(\ell)\otimes\Pc(m)\otimes\Pc(n)\to\Pc(\ell+m+n-2)$.

    We define a map of operads $\Pc^1\to\Pc^2$ in the obvious manner, i.e. it consists of a collection of equivariant maps
    $\Pc^1(n)\to\Pc^2(n)$ for all $n\in\Nb$ which intertwine the various maps $\circ_i$ for $\Pc^1$ and $\Pc^2$.
\end{definition}

\begin{remark}\label{rem:op}
    To make sense of these axioms, it is helpful to picture $\Pc(n)$ as a collection of labels for trees with one outgoing leaf
    and $n$ incoming leaves marked $1,\ldots,n$; the group $S_n$ acts by permuting the markings of the incoming leaves. In this
    picture, the map $\circ_i$ corresponds to gluing the outgoing leaf of a tree in $\Pc(n)$ to the $i^{th}$-incoming leaf of a
    tree in $\Pc(m)$ to obtain a tree in $\Pc(n+m-1)$. Axiom \eqref{opeq} requires that if we first relabel and then glue, this
    is equivalent to gluing first and then relabelling in the natural fashion. Axioms \eqref{opass1} and \eqref{opass2} require
    the gluing of three trees to be associative in the natural fashion.
\end{remark}

Denote by $S_{n+}$ the permutation group on $n+1$ letters $\{0,\ldots,n\}$. Denote by $\tau_n$ the cycle $(01\cdots n)$.
\begin{definition}
    A \emph{cyclic operad} is an operad $\Pc$ in $(\Dc,\otimes,\sigma)$ such that, for each $n\in\Nb$, the $S_n$-action on
    $\Pc(n)$ extends to an $S_{n+}$-action\footnote{Under the embedding $S_n\hookrightarrow S_{n+}$ corresponding to the
    inclusion $\{1,\ldots,n\}\subset\{0,\ldots,n\}$.} and such that
    \begin{equation}
        \tau_{n+m-1}\cdot\circ_m =\circ_1\cdot(\tau_n\otimes\tau_m)\cdot\sigma\label{cyceq}
    \end{equation}
    as maps $\Pc(m)\otimes\Pc(n)\to\Pc(m+n-1)$.

    We define maps of cyclic operads in the obvious manner.
\end{definition}

\begin{remark}
    In the picture of Remark \ref{rem:op}, the objects $\Pc(n)$ of a cyclic operad can be pictured as a collection of labels for
    trees with one outgoing and $n$ incoming leaves, where we are also allowed the permute the outgoing leaf with the incoming
    leaves. Equivalently, we can view $\Pc(n)$ as a collection of labeled trees with $n+1$ leaves marked $0,\ldots,n$, where
    $S_{n+}$ acts by permuting the markings on the leaves. Axiom \eqref{cyceq} requires that if we first relabel using the extra
    symmetry in $S_{n+}$ and then glue, this is equivalent to gluing first and relabelling in the natural fashion.
\end{remark}

\begin{notation}
    If $\Pc$ is a cyclic operad, we write $\Pc((n+1))$ for the object $\Pc(n)$. In this notation, we have
    \begin{equation*}
        \circ_i\colon\Pc((m))\otimes\Pc((n))\to\Pc((m+n-2))
    \end{equation*}
    for $m,n>1$. We will also consider cyclic operads $\Pc$ for which we define $\Pc((0))$. However, we do not assume the
    existence of maps $\circ_i$ with source $\Pc((m))\otimes\Pc((n))$ for either $m$ or $n$ equal to $0$.
\end{notation}

\begin{definition}
    A \emph{stable cyclic operad} is a cyclic operad $\Pc$ such that for each non-negative integer $n\in\Nb$, there exists an
    $S_n$-equivariant decomposition
    \begin{equation*}
        \Pc((n)):=\coprod_{g\in\Nb} \Pc((g,n))
    \end{equation*}
    such that $\Pc((g,n))=0$ if $n<3-2g$, and such that, for all $1\le i\le m$ and $n>0$, the map $\circ_i$ restricts to a map
    \begin{equation*}
        \Pc((g,m))\otimes\Pc((h,n))\to\Pc((g+h,m+n-2)).
    \end{equation*}
\end{definition}

\begin{remark}
    In a stable cyclic operad, we can picture the object $\Pc((g,n))$ as a collection of labels for dual graphs of stable curves
    of genus $g$ with $n$ marked points. In this picture, the maps $\circ_i$ correspond to gluing the first leg of a graph in
    $\Pc((h,n))$ to the $i^{th}$ leg of a graph in $\Pc((g,m))$, and relabelling the remaining legs accordingly.
\end{remark}

\begin{definition}
    Let $n\ge 2$, let $\rho\in S_n$, and let $i\neq j\in \{1,\ldots,n\}$. Denote by $\rho_{\setminus\{i,j\}}\in S_{n-2}$ the
    induced bijection
    \begin{equation*}
        \{1,\ldots,n-2\}\to^\cong\{1,\ldots,n\}\setminus\{i,j\}\to^\rho
        \{1,\ldots,n\}\setminus\{\rho(i),\rho(j)\}\to^\cong\{1,\ldots,n-2\},
    \end{equation*}
    where the first and last bijections are the canonical order-preserving bijections.
\end{definition}

\begin{definition}\label{def:modularoperad}
    A \emph{modular operad} is a stable cyclic operad $\Pc$ such that for each $g$, $n$ and $i\neq j\in\{1,\ldots,n\}$, there
    exists a map
    \begin{equation*}
        \xi_{ij}\colon\Pc((g,n))\to\Pc((g+1,n-2))
    \end{equation*}
    such that the following properties are satisfied. For each $\rho\in S_n$, we require
    \begin{equation}
        \rho_{\setminus\{i,j\}}\cdot\xi_{ij}=\xi_{\rho(i)\rho(j)}\cdot\rho\label{modeq}
    \end{equation}
    as maps $\Pc((g,n))\to\Pc((g+1,n-2))$.

    For $1\le i\ne j\ne k\ne\ell\le n$, we require
    \begin{equation}
        \xi_{ij}\cdot\xi_{k\ell}=\xi_{k\ell}\cdot\xi_{ij}\label{modass1}
    \end{equation}
    as maps $\Pc((g,n))\to\Pc((g+2,n-4))$.\footnote{Note that we are abusing notation slightly on the left hand side of
    \eqref{modass1} by writing $\xi_{ij}$ to denote the map which corresponds to the image of the pair $i,j$ under the
    identification $\{1,\ldots,n\}\setminus\{k,\ell\}\cong\{1,\ldots,n-2\}$. An analogous abuse of notation also occurs on the
    right hand side.}

    We further require
    \begin{align}
        \xi_{12}\cdot\circ_m&=\circ_{m-2}\cdot(\xi_{12}\otimes 1_{\Pc((h,n))})\label{modass2}\\
        \nonumber\\
        \xi_{m,m+1}\cdot\circ_m&=\circ_m\cdot(1_{\Pc((g,m))}\otimes \xi_{12})\label{modass3}\intertext{and}
        \xi_{m-1,m}\cdot\circ_m&=\xi_{m+n-2,m+n-1}\cdot\circ_{m-1}\cdot(1_{\Pc((g,m))}\otimes \tau_n^{-1})\label{modass4}
    \end{align}
    as maps $\Pc((g,m))\otimes\Pc((h,n))\to\Pc((g+h+1,m+n-4))$.

    We define maps of modular operads in the obvious manner.
\end{definition}

\begin{remark}
    A modular operad is a stable cyclic operad with extra structure encoded by the maps $\xi_{ij}$.  If we picture $\Pc((g,n))$
    as a collection of labels for dual graphs of stable marked curves, then the maps $\xi_{ij}$ correspond to gluing together
    the
    $i$ and $j$ legs of a graph $\Gamma$ to obtain a new graph $\Gamma'$. As usual, Axiom \eqref{modeq} requires that
    relabelling
    and then gluing is equivalent to gluing first and relabelling in the natural fashion. Similarly, Axioms
    \eqref{modass1}--\eqref{modass4} require the various operations involving gluing two pairs of legs together to be
    associative
    in the natural fashion.
\end{remark}

\subsubsection*{Proof of Theorem \ref{thm:main}}
    Because the forgetful maps $\oEc_{g,n}^k\to\oMc_{g,n}$ are $S_n$-equivariant, by Corollary \ref{cor:forget}, it suffices to
    verify that the gluing maps on $\{\oEc_{g,n}^k\}$ form a modular operad in order to conclude that the forgetful maps
    \begin{equation*}
        \oEc_{g,n}^k\to\oMc_{g,n}
    \end{equation*}
    determine a map of operads (the analogous observation applies to the cyclic and stable cyclic operads $\oEc_{0,c}^k$ and $\oEc_c^k$).

    Further, the construction of the gluing maps in the proof of Theorem \ref{thm:glue} immediately implies that the equivariance axioms \eqref{opeq}, \eqref{cyceq} and \eqref{modeq} are all satisfied in each case. Therefore, it only remains prove that the gluing maps
    satisfy the associativity properties \eqref{opass1}, \eqref{opass2} and \eqref{modass1}--\eqref{modass4}. These will all be
    immediate consequences of the following lemma.
    \begin{lemma}
        Let $k\ge 5$, let $S$ be a scheme, and let $(\Cc,\{\sigma_i\}_{i=1}^n,\eta)$ be a marked $k$-log-canonically
        embedded curve over $S$. Denote by
        \begin{align*}
            (\Cc^{\gl_{12,34}},\{\sigma_i\}_{i=5}^n,\eta^{\gl_{12,34}})
        \end{align*}
        the embedded curve obtained by first gluing $\sigma_1$ to $\sigma_2$, and then gluing $\sigma_3$ to $\sigma_4$.
        Likewise, denote by
        \begin{align*}
            (\Cc^{\gl_{34,12}},\{\sigma_i\}_{i=5}^n,\eta^{\gl_{34,12}})
        \end{align*}
        the embedded curve obtained by first gluing $\sigma_3$ to $\sigma_4$, and then gluing $\sigma_1$ to $\sigma_2$. Then
        there exists a canonical isomorphism
        \begin{equation*}
            (\Cc^{\gl_{12,34}},\{\sigma_i\}_{i=5}^n,\eta^{\gl_{12,34}})\cong(\Cc^{\gl_{34,12}},\{\sigma_i\}_{i=5}^n,\eta^{\gl_{34,12}}).
        \end{equation*}
        The same conclusion holds if $k\ge 2$ and $\{\sigma_1,\sigma_2\},~\{\sigma_3,\sigma_4\}$ lie on disjoint components of $\Cc$ and $\Cc^{\gl}$.  In addition, if all components of $\Cc$ have arithmetic genus 0, then we can take $k\ge 1$.
    \end{lemma}
    \begin{proof}
        The associativity of the classical gluing maps for curves guarantees the existence of a canonical isomorphism
        \begin{equation*}
            (\Cc^{\gl_{12,34}},\{\sigma_i\}_{i=5}^n)\cong(\Cc^{\gl_{34,12}},\{\sigma_i\}_{i=5}^n).
        \end{equation*}
        Using this isomorphism, our observations about the vanishing of higher direct image sheaves imply that there exists a
        commuting diagram of $\Oc_S$-modules with exact rows and columns:
        \begin{equation*}
            \begin{xy}
                \square(0,500)<1750,500>[\pi_\ast\omega_{\Cc^{\gl_{12,34}}/S}\left(\sum_{i=5}^n\sigma_i\right)^{\otimes
                k}`\pi_\ast\omega_{\Cc^{\gl_{34}}/S}\left(\sum_{i\neq 3,4}\sigma_i\right)^{\otimes
                k}`\pi_\ast\omega_{\Cc^{\gl_{12}}/S}\left(\sum_{i=3}^n\sigma_i\right)^{\otimes
                k}`\pi_\ast\omega_{\Cc/S}\left(\sum_{i=1}^n\sigma_i\right)^{\otimes k};```]
                \square(1750,500)/>`>`=`>/<1000,500>[\pi_\ast\omega_{\Cc^{\gl_{34}}/S}\left(\sum_{i\neq
                3,4}\sigma_i\right)^{\otimes k}`\Oc_S`\pi_\ast\omega_{\Cc/S}\left(\sum_{i=1}^n\sigma_i\right)^{\otimes
                k}`\Oc_S;```]
                \square(0,0)/>`>`>`=/<1750,500>[\pi_\ast\omega_{\Cc^{\gl_{12}}/S}\left(\sum_{i=3}^n\sigma_i\right)^{\otimes
                k}`\pi_\ast\omega_{\Cc/S}\left(\sum_{i=1}^n\sigma_i\right)^{\otimes k}`\Oc_S`\Oc_S;```]
                \square(1750,0)<1000,500>[\pi_\ast\omega_{\Cc/S}\left(\sum_{i=1}^n\sigma_i\right)^{\otimes k}`\Oc_S`\Oc_S`0;```]
            \end{xy}.
        \end{equation*}
        Dualizing and projectivizing, we obtain a commuting diagram
        \begin{equation*}
            \begin{xy}
                \square(0,500)<1000,500>[\emptyset`S`S`\Pb_S\left(\pi_\ast\omega_{\Cc/S}\left(\sum_{i=1}^n\sigma_i\right)^{\otimes
                k}\right)^\vee;```]
                \square(1000,500)/=`>`>`-->/<1800,500>[S`S`\Pb_S\left(\pi_\ast\omega_{\Cc/S}\left(\sum_{i=1}^n\sigma_i\right)^{\otimes
                k}\right)^\vee`\Pb_S\left(\pi_\ast\omega_{\Cc^{\gl_{12}}/S}\left(\sum_{i=3}^n\sigma_i\right)^{\otimes k}\right)^\vee;```]
                \square(0,0)/>`=`-->`>/<1000,500>[S`\Pb_S\left(\pi_\ast\omega_{\Cc/S}\left(\sum_{i=1}^n\sigma_i\right)^{\otimes
                k}\right)^\vee`S`\Pb_S\left(\pi_\ast\omega_{\Cc^{\gl_{34}}/S}\left(\sum_{i\neq 3,4}\sigma_i\right)^{\otimes k}\right)^\vee;```]
                \square(1000,0)/-->`-->`-->`-->/<1800,500>[\Pb_S\left(\pi_\ast\omega_{\Cc/S}\left(\sum_{i=1}^n\sigma_i\right)^{\otimes
                k}\right)^\vee`\Pb_S\left(\pi_\ast\omega_{\Cc^{\gl_{12}}/S}\left(\sum_{i=3}^n\sigma_i\right)^{\otimes
                k}\right)^\vee`\Pb_S\left(\pi_\ast\omega_{\Cc^{\gl_{34}}/S}\left(\sum_{i\neq 3,4}\sigma_i\right)^{\otimes
                k}\right)^\vee`\Pb_S\left(\pi_\ast\omega_{\Cc^{\gl_{12,34}}/S}\left(\sum_{i=5}^n\sigma_i\right)^{\otimes k}\right)^\vee;```]
            \end{xy}
        \end{equation*}
        where the dashed arrows indicate the projections. The commutativity of the lower right square implies that
        $\eta^{\gl_{12,34}}=\eta^{\gl_{34,12}}$.
    \end{proof}

\begin{remark}
    There is an equivalent, although more manifestly geometric, formulation of the above argument.  Each of the pairs of points
    $\{\sigma_1,\sigma_2\}$ and $\{\sigma_3,\sigma_4\}$ lying over eventual nodes gives a point on the line between them.  Then,
    projection from one point followed by projection from the image of the other, in either order, is the same map as projection
    from the line spanned by the two points.  Here, the equivalence of the embeddings follows from the fact that these are just
    two factorizations of the same projection map, with one-dimensional center.
\end{remark}

\bibliographystyle{alpha}
\bibliography{Bibliography}
\end{document}